\documentclass[journal,onecolumn,12pt]{article}

\usepackage[a4paper,margin=1in]{geometry}
\usepackage{amsmath,amssymb,amsthm}
\usepackage{mathtools}
\usepackage{caption2}
\usepackage{multirow}
\usepackage{graphicx}
\usepackage{enumerate}
\usepackage{bm}
\usepackage{tikz}
\usepackage[hidelinks]{hyperref}
\allowdisplaybreaks[4]
\begin{document}
\title{On the weighted hard-core model and Rado's covering problem for congruent Euclidean balls\thanks{This research was supported by the National Key Research and Development Program of China under Grant 2025YFC3409900, the National Natural Science Foundation of China under Grant 12231014, and Beijing Scholars Program. The work of C. Xie was supported by the National Natural Science Foundation of China under Grant No. 12401440.}}
\author{Chengfei Xie\thanks{C. Xie is with the Institute of Mathematics and Interdisciplinary Sciences, Xidian University, Xi'an 710126, China (e-mail: cfxie@cnu.edu.cn).}
 ~and Gennian Ge\thanks{G. Ge is with the School of Mathematical Sciences, Capital Normal University, Beijing 100048, China (e-mail: gnge@zju.edu.cn).}}

\newtheorem{theorem}{Theorem}[section]
\newtheorem{lemma}[theorem]{Lemma}
\newtheorem{proposition}[theorem]{Proposition}
\newtheorem{corollary}[theorem]{Corollary}
\newtheorem{conjecture}[theorem]{Conjecture}
\newtheorem{definition}[theorem]{Definition}
\newtheorem{construction}[theorem]{Construction}
\newtheorem{remark}[theorem]{Remark}
\newtheorem{claim}[theorem]{Claim}
\newtheorem{example}[theorem]{Example}
\newtheorem{question}[theorem]{Question}

\maketitle
\begin{abstract}
Let $K$ be a symmetric convex body in $\mathbb{R}^d$ and let $f(K)$ denote the largest constant $c$ such that every finite collection of translates of $K$ contains a pairwise disjoint subcollection whose total volume is at least $c$ times the volume of the union of the original collection. The classical Vitali covering lemma gives $f(K)\geq3^{-d}$. In this paper, we establish two improvements. First, by a purely combinatorial argument, we prove that
$$
f(K)\geq \frac{2}{3^d + 2^d}
$$
for every symmetric convex body $K$. This improves the Vitali bound by a factor tending to $2$ as $d$ tends to infinity. Second, using a weighted hard-core model together with a weighted geometric estimate for intersections of Euclidean balls, we show that, for all sufficiently large $d$,
$$
f(B^d)\geq
 \left(
   \log\frac{3}{1+\sqrt3}
   -O\left(\frac{\log d}{d}\right)
 \right)d\,3^{-d},
 $$
where $B^d$ is the unit Euclidean ball in $\mathbb{R}^d$. Thus, the classical lower bound is improved by a factor of order $d$.
\smallskip
\end{abstract}
\medskip


\smallskip


\section{Introduction}
In 1928, Rad\'{o} observed that for every finite collection of compact intervals that covers the unit interval, there is a subcollection consisting of pairwise disjoint intervals with total length at least $1/2$ and that this constant is sharp \cite{bwmeta1.element.bwnjournal-article-fmv11i1p25bwm}. In 1949, Rado generalized problems of this type and introduced the following definition \cite{MR30782}.
\begin{definition}\label{bigf}
Let $K$ be a convex body. Let $F(K)$ denote the largest constant $c\geq0$ such that every finite collection $\mathcal{C}$ of homothetic copies of $K$ admits a pairwise disjoint subcollection $\mathcal{S}$ satisfying
$$
    \operatorname{vol}\left(\bigcup_{K'\in\mathcal S}K'\right)
    \geq
    c\operatorname{vol}\left(\bigcup_{K'\in\mathcal{C}}K'\right).
$$
\end{definition}
Let $Q^d=[0,1 ]^d$ be the $d$-dimensional cube. Rad\'{o}'s result then shows that $F(Q^1)=1/2$. In the same paper, Rad\'{o} also proved that $F(Q^2)\geq1/9$ and conjectured the sharp bound $F(Q^2)=1/4$. Remarkably, in 1973, Ajtai disproved this conjecture \cite{MR319053}:
$$
F(Q^2)<1/4.
$$

The celebrated Vitali covering lemma states that in an arbitrary metric space, for every finite collection $\mathcal{C}$ of balls, there exists a pairwise disjoint subcollection $\mathcal{S}$ such that
$$
    \bigcup_{B\in\mathcal{C}}B
    \subseteq
    \bigcup_{B\in\mathcal S}3B,
$$
where $3B$ is the ball with the same center and three times the radius of $B$. We refer readers to \cite{MR924157} for a proof. Therefore, we have
$$
F(K)\geq3^{-d}
$$
for every symmetric convex body $K$.

The following variant of Definition \ref{bigf} was also introduced by Rado \cite{MR30782}.
\begin{definition}
Let $K$ be a convex body. Let $f(K)$ denote the largest constant $c\geq 0$ such that every finite collection $\mathcal{C}$ of translates of $K$ admits a pairwise disjoint subcollection $\mathcal{S}$
satisfying
$$
    \operatorname{vol}\left(\bigcup_{K'\in\mathcal S}K'\right)
    \geq
    c\operatorname{vol}\left(\bigcup_{K'\in\mathcal{C}}K'\right).
$$
\end{definition}
\begin{remark}
Although both $F(K)$ and $f(K)$ can be defined for general convex bodies, we restrict throughout to symmetric convex bodies.
\end{remark}
Clearly,
$$
f(K)\geq F(K).
$$
By considering all translates of $K$ containing the origin, we have
\begin{equation}\label{upperfk}
  f(K)\leq2^{-d}.
\end{equation}
In \cite{MR30782}, Rado proved that the upper bound is sharp for cubes:
$$
  f(Q^d)=2^{-d}.
$$
We refer readers to \cite{2026arXiv260417509D} for a recent survey of this topic.

In this paper, we focus on the lower bound for $f(K)$. Our first result is a constant factor improvement.
\begin{theorem}\label{ruo}
For every symmetric convex body $K\subseteq\mathbb{R}^d$,
$$
        f(K)\geq \frac{2}{3^d+2^d}.
$$
Equivalently, every finite collection $\mathcal{C}$ of translates of $K$ contains a pairwise disjoint subcollection
$\mathcal S$ such that
$$
    \operatorname{vol}\left(\bigcup_{K'\in\mathcal S}K'\right)
    \geq
    \frac{2}{3^d+2^d}
    \operatorname{vol}\left(\bigcup_{K'\in\mathcal{C}}K'\right).
$$
\end{theorem}
\begin{remark}
Note that
$$
    \frac{6}{5}\cdot3^{-d}
    \leq
    \frac{2}{3^d+2^d}
    =
    \frac{2}{1+(2/3)^d}3^{-d}
    =
    (2-o_d(1))3^{-d}.
$$
So Theorem \ref{ruo} improves the Vitali lower bound by a constant factor that is at least $6/5$ and tends to $2$ as $d$ tends to infinity.
\end{remark}

Our second result concerns Euclidean balls. Let $B^d$ denote the Euclidean unit ball in $\mathbb{R}^d$. By \eqref{upperfk} and Theorem \ref{ruo}, we have
$$
\frac{2}{3^d+2^d}\leq f(B^d)\leq2^{-d}.
$$

The Kabatjanski\u{\i}-Leven\v{s}te\u{\i}n upper bound for spherical codes \cite{MR0514023} yields an exponential improvement for the upper bound $f(B^d)\leq O(2.447^{-d})$; see \cite{2026arXiv260417509D} for a proof. As for the lower bound, Rado \cite{MR30782} proved that
$$
f(B^2)\geq\frac{\pi}{8\sqrt3}.
$$
In high dimensions, we improve upon Theorem \ref{ruo} by a factor of order $d$ for $f(B^d)$.
\begin{theorem}\label{qiang}
For sufficiently large $d$,
$$
 f(B^d)
 \ge
 \left(
   \log\frac{3}{1+\sqrt3}
   -O\left(\frac{\log d}{d}\right)
 \right)d\,3^{-d}.
$$
\end{theorem}

To the best of our knowledge, for arbitrary symmetric convex bodies, no improvement over the bound $f(K)\geq3^{-d}$ had previously been obtained. Thus, Theorem \ref{ruo} provides the first constant-factor improvement, while Theorem 1.6 goes substantially further for Euclidean balls, improving the classical Vitali lower bound by a factor of order $d$.

The proofs of Theorems 1.4 and 1.6 rely on fundamentally different ideas.
Theorem 1.4 follows from a purely combinatorial analysis of a maximum independent
set in the intersection graph. To prove Theorem 1.6, we partition the union of the balls into disjoint measurable subsets, each assigned to one of the balls, and introduce a weighted
hard-core model with vertex fugacities proportional to the normalized volumes of these subsets. This weighting allows the model to capture
both the combinatorial structure of the intersection graph and the geometric contribution of each ball. Together with a weighted geometric estimate
for intersections of Euclidean balls, this approach yields the dimension-dependent improvement in Theorem 1.6. Building on the hard-core model framework developed by Jenssen, Joos, and Perkins \cite{MR3836667, MR3898718}, our approach introduces a weighted formulation adapted to the nonuniform geometric setting arising
in Rado's covering problem.
\section{Preliminaries}
Let $K\subseteq\mathbb{R}^d$ be a symmetric convex body.
For a point $x\in\mathbb{R}^d$ and a scalar $r>0$, we define
$$
K+x=\{y+x:y\in K\}\quad\text{and}\quad rK=\{rx:x\in K\}.
$$
Let $\|\cdot\|_K$ denote the Minkowski functional with respect to $K$; that is,
$$
\|x\|_K=\inf\{r>0:x\in rK\}.
$$

For a finite collection $\mathcal{C}$ of translates of $K$, let $\alpha(\mathcal{C})$ denote the maximum cardinality of a pairwise disjoint subcollection of $\mathcal{C}$, where two translates are regarded as intersecting if they intersect at boundaries. We have
$$
 f(K)
 =
 \inf_{\mathcal{C}}
 \frac{\alpha(\mathcal{C})\operatorname{vol}(K)}
      {\operatorname{vol}\left(\bigcup_{K'\in\mathcal{C}}K'\right)},
$$
where the infimum is over all nonempty finite collections of translates of $K$.

Given a finite set $V=\{v_1, v_2, \ldots, v_n\}$ of $n$ points in $\mathbb{R}^d$, let $\mathcal{C}=\{K+v_i:1\leq i\leq n\}$ be a collection of $n$ translates of $K$ and write
$$
U=\bigcup_{i=1}^n\left(K+v_i\right).
$$
Let $G=(V, E)$ be the intersection graph of $\mathcal{C}$. That is, for distinct vertices $v_i, v_j\in V$, $\{v_i, v_j\}\in E$ if and only if $\left(K+v_i\right)\cap \left(K+v_j\right)\neq\emptyset$, or equivalently, $\|v_i-v_j\|_K\leq2$.

We partition $U$ into disjoint measurable subsets as follows. Define
$$
A_1=K+v_1
$$
and for $2\leq k\leq n$,
$$
A_k=\left(K+v_k\right)\setminus\bigcup_{i=1}^{k-1}\left(K+v_i\right).
$$
Then the sets $A_i$ are pairwise disjoint, each $A_i$ is contained in $K+v_i$, and $U=\bigsqcup_{i=1}^nA_i$, where $\bigsqcup$ means the disjoint union.

Define
$$
a_i=\frac{\operatorname{vol}(A_i)}{\operatorname{vol}(K)}\quad\text{and}\quad M=\sum_{i=1}^na_i.
$$
So
\begin{equation}\label{eq1}
  M=\sum_{i=1}^na_i=\sum_{i=1}^n\frac{\operatorname{vol}(A_i)}{\operatorname{vol}(K)}=\frac{\operatorname{vol}(U)}{\operatorname{vol}(K)}.
\end{equation}

For a vertex $v\in V$, let $N(v)$ and $N[v]$ be its open and closed neighborhoods, respectively; that is,
$$
N(v)=\{u\in V:\{u, v\}\in E\}\quad\text{and}\quad N[v]=N(v)\cup\{v\}.
$$

Throughout the paper, $\log$ denotes the natural logarithm.

\section{Proof of Theorem \ref{ruo}}
As a warm-up, we give a purely combinatorial proof of Theorem \ref{ruo} in this section without the hard-core model.

Let $\alpha$ be the independence number of $G$ and choose an independent set $I\subseteq V$ such that $|I|=\alpha$. $\{K+v:v\in I\}$ is a pairwise disjoint subcollection, and thus
$$
\operatorname{vol}\left(\bigcup_{v\in I}(K+v)\right)=\alpha\operatorname{vol}(K).
$$

Define
$$
n(v)=|N[v]\cap I|.
$$
If $n(v)=0$, then $\{v\}\cup I$ is an independent set of size $\alpha+1>\alpha$, a contradiction. So
\begin{equation}\label{eq2}
  n(v)\geq1\text{ for every }v\in V.
\end{equation}

We split $M$ into two parts:
$$
P_1=\sum_{\substack{1\leq i\leq n\\ n(v_i)=1}}a_i,\qquad P_2=\sum_{\substack{1\leq i\leq n\\ n(v_i)\geq2}}a_i.
$$
By \eqref{eq2}, $M=P_1+P_2$. Our goal is to derive upper bounds for $P_1$ and $P_2$, which will imply an upper bound for $M$.

\subsection{An upper bound for $P_1$}
For every $v\in I$, we define
$$
Q_v=\{u\in V:N[u]\cap I=\{v\}\}
$$
and let $q_v=\sum_{v_i\in Q_v}a_i$. For every $u\in V$ with $n(u)=1$, $u$ belongs to exactly one of the sets $Q_v$. It follows that
\begin{equation}\label{eq3}
  P_1=\sum_{v\in I}q_v.
\end{equation}
\begin{lemma}\label{clique}
For every $v\in I$, $Q_v$ induces a clique in $G$.
\end{lemma}
\begin{proof}
Fix $v\in I$. Suppose, to the contrary, that there are distinct vertices $u, w\in Q_v$ that are not adjacent.

Since $u, w\in Q_v$, we have
$$
N[u]\cap I=N[w]\cap I=\{v\}.
$$
Thus neither $u$ nor $w$ is adjacent to any vertex of
$I\setminus\{v\}$. By assumption, $u$ and $w$ are also not adjacent
to each other. Consequently,
$$
(I\setminus\{v\})\cup\{u,w\}
$$
is an independent set with cardinality $|I|-1+2=\alpha+1$,
contradicting the maximality of $I$.
\end{proof}
\begin{lemma}\label{qvbound}
For every $v\in I$, $q_v\leq2^d$.
\end{lemma}
\begin{proof}
Fix $v\in I$ and let
$$
E_v=\bigcup_{v_i\in Q_v}A_i.
$$
Then
$$
q_v=\sum_{v_i\in Q_v}a_i=\sum_{v_i\in Q_v}\frac{\operatorname{vol}(A_i)}{\operatorname{vol}(K)}=\frac{\operatorname{vol}(E_v)}{\operatorname{vol}(K)},
$$
where the last equation follows from the pairwise disjointness of the sets $A_i$.

We claim that $E_v-E_v\subseteq4K$, where $E_v-E_v=\{u-w:u, w\in E_v\}$. For every $u, w\in E_v$, there exist
$v_i,v_j\in Q_v$ such that
$$
    u\in A_i\subseteq K+v_i
    \text{ and }
    w\in A_j\subseteq K+v_j.
$$
By Lemma \ref{clique}, the vertices $v_i$ and $v_j$ are adjacent unless they coincide. In either case,
$$
    \|v_i-v_j\|_K\leq 2.
$$
It follows from the triangle inequality that
\begin{align*}
    \|u-w\|_K
    &\leq\|u-v_i\|_K+\|v_i-v_j\|_K+\|v_j-w\|_K\\
    &\leq 1+2+1=4.
\end{align*}
Hence $E_v-E_v\subseteq4K$.

By the Brunn-Minkowski inequality (see, for example, \cite[Theorem 8.1.1]{MR936419}),
$$
\operatorname{vol}(E_v-E_v)\geq2^d\operatorname{vol}(E_v).
$$
Therefore
$$
\operatorname{vol}(E_v)\leq 2^{-d}\operatorname{vol}(E_v-E_v)\leq2^{-d}\operatorname{vol}(4K)=2^d\operatorname{vol}(K),
$$
and $q_v\leq2^d$.
\end{proof}
Summing the inequality in Lemma \ref{qvbound} over $v\in I$ and using
\eqref{eq3}, we obtain
\begin{equation}\label{eq4}
    P_1\leq 2^d\alpha.
\end{equation}
\subsection{An upper bound for $P_2$}
For every $v\in I$, define
$$
m_v=\sum_{\substack{v_i\in V\\v_i\in N[v]}}a_i=\sum_{\substack{v_i\in V\\v_i\in N[v]}}\frac{\operatorname{vol}(A_i)}{\operatorname{vol}(K)}=\frac{1}{\operatorname{vol}(K)}\operatorname{vol}\left(\bigsqcup_{\substack{v_i\in V\\v_i\in N[v]}}A_i\right).
$$
We have the following lemma.
\begin{lemma}\label{mxbound}
For every $v\in I$, $m_v\leq3^d$.
\end{lemma}
\begin{proof}
Fix $v\in I$. For every $v_i\in N[v]$, the vertices $v_i$ and $v$ are adjacent unless they coincide. In either case,
$$
    \|v_i-v\|_K\leq 2.
$$
If $u\in A_i$, then $u\in K+v_i$, and hence
$$
\|u-v\|_K\leq\|u-v_i\|_K+\|v_i-v\|_K\leq 1+2=3.
$$
Thus $A_i\subseteq 3K+v$ whenever $v_i\in N[v]$. Therefore
$$
m_v=\frac{1}{\operatorname{vol}(K)}\left(\bigsqcup_{\substack{v_i\in V\\v_i\in N[v]}}\operatorname{vol}(A_i)\right)\leq\frac{\operatorname{vol}(3K+v)}{\operatorname{vol}(K)}=3^d.
$$
\end{proof}
We now give an upper bound for $P_2$.
\begin{lemma}\label{p2bound}
$$
P_2\leq\frac{1}{2}\left(3^d\alpha-P_1\right).
$$
\end{lemma}
\begin{proof}
For every $v\in I$, recall that
$$
m_v=\sum_{\substack{v_i\in V\\v_i\in N[v]}}a_i=\sum_{\substack{v_i\in V\\v\in N[v_i]}}a_i
$$
and
$$
q_v=\sum_{v_i\in Q_v}a_i=\sum_{\substack{v_i\in V\\N[v_i]\cap I=\{v\}}}a_i.
$$
So
$$
m_v-q_v=\sum_{\substack{v_i\in V\\N[v_i]\cap I\supsetneqq\{v\}}}a_i=\sum_{\substack{v_i\in V\\v\in N[v_i]\\|N[v_i]\cap I|\geq2}}a_i=\sum_{\substack{v_i\in V\\v\in N[v_i]\\n(v_i)\geq2}}a_i.
$$
Summing over $v\in I$, we obtain
\begin{align*}
   \sum_{v\in I}(m_v-q_v)
   &=\sum_{v\in I}\sum_{\substack{v_i\in V\\v\in N[v_i]\\n(v_i)\geq2}}a_i\\
   &=\sum_{\substack{v_i\in V\\n(v_i)\geq2}}\sum_{v\in N[v_i]\cap I}a_i\\
   &=\sum_{\substack{v_i\in V\\n(v_i)\geq2}}n(v_i)a_i\\
   &\geq\sum_{\substack{v_i\in V\\n(v_i)\geq2}}2a_i=2P_2.
\end{align*}
On the other hand, by Lemma \ref{mxbound}, $\sum_{v\in I}m_v\leq\sum_{v\in I}3^d=3^d\alpha$, while \eqref{eq3} yields $\sum_{v\in I}q_v=P_1$. Combining these yields
$$
P_2\leq\frac{1}{2}\sum_{v\in I}(m_v-q_v)\leq\frac{1}{2}\left(3^d\alpha-P_1\right).
$$
\end{proof}
\subsection{Proof of Theorem \ref{ruo}}
We are ready to prove Theorem \ref{ruo}.
\begin{proof}[Proof of Theorem \ref{ruo}]
By Lemma \ref{p2bound} and \eqref{eq4}, we have
$$
  M  =P_1+P_2 \leq\frac{1}{2}\left(3^d\alpha+P_1\right)\leq\frac{3^d+2^d}{2}\alpha.
$$
Recalling \eqref{eq1}, we have
$$
\operatorname{vol}(U)=M\operatorname{vol}(K)\leq\frac{3^d+2^d}{2}\alpha\operatorname{vol}(K).
$$
The independent set $I$ corresponds to $\alpha$ pairwise disjoint translates of $K$, whose total volume is $\alpha\operatorname{vol}(K)$. Therefore,
$$
\frac{\alpha\operatorname{vol}(K)}{\operatorname{vol}(U)}\geq\frac{2}{3^d+2^d}.
$$
Since this holds for every finite collection $\mathcal{C}$ of translates of $K$, it follows that $f(K)\geq\frac{2}{3^d+2^d}$.
\end{proof}
\section{The weighted hard-core model}
Let
$$
    B^d(x, r)=\{y\in\mathbb{R}^d:\|y-x\|\leq r\}
$$
be the closed Euclidean ball in $\mathbb{R}^d$ centered at $x$ with radius $r$, where $\|\cdot\|$ is the Euclidean norm. We denote by $\omega_d$ the volume of the unit ball. From now on, we assume that $\mathcal{C}=\{B^d(v_i,1):1\leq i\leq n\}$ is a collection of $n$ unit balls.
Recall that
$$
a_i=\frac{\operatorname{vol}(A_i)}{\omega_d}.
$$
We may delete vertices with $a_i=0$. This does not change $M$,
and the independence number of $G$ is at least the
independence number of the graph induced by the remaining vertices.
Therefore we also assume that $a_i>0$ for every $v_i\in V$.

Fix $\lambda>0$. For an independent set $I\in\mathcal{I}(G)$, define its weight by
$$
 w_\lambda(I)=\lambda^{|I|}\prod_{v_i\in I}a_i=\prod_{v_i\in I}(\lambda a_i).
$$
The \emph{weighted hard-core model} on $G$ with vertex fugacities $\{\lambda a_i\}_{1\leq i\leq n}$ is a random independent set $\mathbf{X}$ of $G$, where any independent set $I\in\mathcal{I}(G)$ is chosen with probability proportional to its weight; that is, the probability of choosing $I\in\mathcal{I}(G)$ is
$$
\mathbb{P}_\lambda(\mathbf{X}=I)= \frac{\lambda^{|I|}\prod_{v_i\in I}a_i}{Z_G(\lambda)},
$$
where
$$
Z_G(\lambda)= \sum_{I\in\mathcal{I}(G)}\lambda^{|I|}\prod_{v_i\in I}a_i
$$
is the weighted partition function.

Let $\rho=\mathbb{E}_\lambda|\mathbf{X}|$ and $\tilde{\rho}=\frac{\rho}{M}$. Every realization of $\mathbf{X}$ is an independent set; therefore
$$
\rho\leq\alpha(G).
$$
It follows from \eqref{eq1} that
\begin{equation}\label{eqp}
  \frac{\alpha(G)\omega_d}{\operatorname{vol}(U)}=\frac{\alpha(G)}{M}\geq \tilde{\rho}.
\end{equation}
It is therefore enough to obtain a uniform lower bound for $\tilde{\rho}$.
Moreover, we calculate
\begin{equation}\label{eq5}
\begin{split}
  \rho &=\mathbb{E}_\lambda|\mathbf{X}|  \\
   & =\sum_{I\in\mathcal{I}(G)}|I|\frac{\lambda^{|I|}\prod_{v_i\in I}a_i}{Z_G(\lambda)}\\
   &=\lambda\frac{Z'_G(\lambda)}{Z_G(\lambda)}\\
   &=\lambda\left(\log Z_G(\lambda)\right)'.
   \end{split}
\end{equation}
\section{A lower bound on the expected occupancy}
Independently of $\mathbf{X}$, choose a random vertex $\mathbf{R}\in V$ according to
\begin{equation}\label{eq6}
 \mathbb{P}(\mathbf{R}=v_i)=\frac{a_i}{M}.
\end{equation}
Given the random independent set $\mathbf{X}$ and the random vertex $\mathbf{R}$ as above, we define a random set
$$
\mathbf{Y}=\mathbf{Y}(\mathbf{X}, \mathbf{R}):=\mathbf{X}\cap(V\setminus N[\mathbf{R}]),
$$
and define the set of locally available vertices by
$$
\mathbf{T}(\mathbf{R}, \mathbf{Y})=\{u\in N[\mathbf{R}]:N(u)\cap \mathbf{Y}=\emptyset\}.
$$
If $\mathbf{R}=v_i$, then $v_i\in \mathbf{T}(\mathbf{R},\mathbf{Y})$, since no vertex of $\mathbf{Y}$ is adjacent to
$v_i$.

For a set $T\subseteq V$, write
$$
Z_T(\lambda)=\sum_{J\in\mathcal{I}(G[T])}\lambda^{|J|}\prod_{v_j\in J}a_j
$$
and
$$
\rho_T(\lambda)=\frac{\lambda Z_T'(\lambda)}{Z_T(\lambda)}.
$$
By \eqref{eq5}, $\rho_T(\lambda)$ is the expected cardinality of an independent set in the
weighted hard-core model on $G[T]$.
\begin{lemma}[Spatial Markov property]
\label{lem:spatial-markov}
Fix a vertex $v\in V$. Conditioned on $\mathbf{Y}=\mathbf{X}\cap(V\setminus N[v])$, the random set $\mathbf{X}\cap N[v]$ has the weighted hard-core
distribution on $G[\mathbf{T}(v,\mathbf{Y})]$, with vertex fugacities $\{\lambda a_i\}_{v_i\in \mathbf{T}(v,\mathbf{Y})}$. In particular, for every $J\in\mathcal{I}(G[\mathbf{T}(v,\mathbf{Y})])$,
\begin{equation}\label{eq7}
\mathbb{P}_\lambda\left(\mathbf{X}\cap N[v]=J|\mathbf{X}\cap(V\setminus N[v])=\mathbf{Y}\right)=\frac{\lambda^{|J|}\prod_{v_j\in J}a_j}{Z_{\mathbf{T}(v,\mathbf{Y})}(\lambda)}.
\end{equation}
\end{lemma}
\begin{proof}
Fix $v\in V$. Suppose $\mathbf{X}\in \mathcal{I}(G)$ and $\mathbf{Y}=\mathbf{X}\cap(V\setminus N[v])$. Then $\mathbf{X}\setminus \mathbf{Y}$ is an independent set and $N(u)\cap \mathbf{Y}=\emptyset$ for every $u\in \mathbf{X}\setminus \mathbf{Y}$. So $\mathbf{X}\setminus \mathbf{Y}\in\mathcal{I}(G[\mathbf{T}(v,\mathbf{Y})])$. Conversely, if $\mathbf{Y}\in\mathcal{I}(G[V\setminus N[v]])$ and $J\in\mathcal{I}(G[\mathbf{T}(v,\mathbf{Y})])$, then $\mathbf{Y}\cup J\in \mathcal{I}(G)$. Thus, conditioned on $\mathbf{Y}=\mathbf{X}\cap(V\setminus N[v])$, the only remaining constraint is that the
vertices inside $\mathbf{T}(v,\mathbf{Y})$ form an independent set.

So we have
\begin{align*}
  &\mathbb{P}_\lambda\left(\mathbf{X}\cap N[v]=J|\mathbf{X}\cap(V\setminus N[v])=\mathbf{Y}\right)\\
 =&\frac{\mathbb{P}_\lambda\left(\mathbf{X}\cap N[v]=J, \mathbf{X}\cap(V\setminus N[v])=\mathbf{Y}\right)}{\mathbb{P}_\lambda\left(\mathbf{X}\cap(V\setminus N[v])=\mathbf{Y}\right)}\\
 =&\frac{\mathbb{P}_\lambda\left(\mathbf{X}=\mathbf{Y}\cup J\right)}{\mathbb{P}_\lambda\left(\mathbf{X}\cap(V\setminus N[v])=\mathbf{Y}\right)}\\
 =&\frac{\lambda^{|\mathbf{Y}\cup J|}\prod_{v_i\in \mathbf{Y}\cup J}a_i}{Z_G(\lambda)\cdot\mathbb{P}_\lambda\left(\mathbf{X}\cap(V\setminus N[v])=\mathbf{Y}\right)}\\
 =&\left(\frac{\lambda^{|\mathbf{Y}|}\prod_{v_k\in \mathbf{Y}}a_k}{Z_G(\lambda)\cdot\mathbb{P}_\lambda\left(\mathbf{X}\cap(V\setminus N[v])=\mathbf{Y}\right)}\right)\left(\lambda^{|J|}\prod_{v_j\in J}a_j\right).
\end{align*}
Conditioned on $\mathbf{Y}$, the first factor is constant and cancels from the normalization. This gives
\eqref{eq7}.
\end{proof}
\begin{lemma}\label{lem:first-local}
For the random set $\mathbf{T}=\mathbf{T}(\mathbf{R}, \mathbf{Y})$,
\begin{equation}\label{eq8}
  \tilde{\rho}=\lambda\cdot\mathbb{E}\left(\frac{1}{Z_\mathbf{T}(\lambda)}\right).
\end{equation}
Consequently, if
$$
z=\mathbb{E}\log Z_\mathbf{T}(\lambda),
$$
then
\begin{equation}\label{eq9}
  \tilde{\rho}\geq\lambda e^{-z}.
\end{equation}
Both expectations above are with respect to the random set $\mathbf{T}$ generated by $\mathbf{X}$ and $\mathbf{R}$.
\end{lemma}
\begin{proof}
Conditioned on $\mathbf{R}=v_i\in V$, let
$$
\mathcal{S}_i=\{I\in \mathcal{I}(G):I\cap N[v_i]=\emptyset\},
$$
and
$$
\mathcal{S}'_i=\{I\in \mathcal{I}(G):v_i\in I\}.
$$
Then the map $f:\mathcal{S}_i\rightarrow\mathcal{S}'_i$ defined by $f(I)=I\cup\{v_i\}$ is a bijection. Moreover, according to the weighted hard-core distribution, for every $I\in\mathcal{S}_i$,
$$
\mathbb{P}_\lambda(\mathbf{X}=f(I))=\lambda a_i\mathbb{P}_\lambda(\mathbf{X}=I).
$$
If we write
$$
b_i=\mathbb{P}_\lambda(\mathbf{X}\in\mathcal{S}_i)=\mathbb{P}_\lambda(\mathbf{X}\cap N[v_i]=\emptyset),
$$
then
$$
\mathbb{P}_\lambda(v_i\in \mathbf{X})=\mathbb{P}_\lambda(\mathbf{X}\in\mathcal{S}'_i)=\lambda a_ib_i.
$$
Summing over $v_i\in V$ gives
$$
 \rho
 =
 \sum_{v_i\in V}\mathbb{P}_\lambda(v_i\in \mathbf{X})
 =
 \lambda\sum_{v_i\in V}a_ib_i
 =
 \lambda\sum_{i=1}^na_ib_i.
$$
After dividing by $M$ and using
\eqref{eq6}, we obtain
$$
 \tilde{\rho}
 =
 \lambda\cdot\mathbb{E}_\mathbf{R} b_\mathbf{R},
$$
where $\mathbb{E}_\mathbf{R}$ is the expectation with respect to the random vertex $\mathbf{R}$.

By Lemma \ref{lem:spatial-markov}, conditioned on $\mathbf{R}=v_i$ and $\mathbf{Y}$, the probability that $\mathbf{X}\cap N[v_i]=\emptyset$ is
$$
\frac{1}{Z_\mathbf{T}(\lambda)}.
$$
Thus
$$
 \tilde{\rho}
 =
 \lambda\cdot\mathbb{E}_\mathbf{R} b_\mathbf{R}
 =
 \lambda\cdot\mathbb{E}_\mathbf{R}\left(\mathbb{P}_\lambda(\mathbf{X}\cap N[\mathbf{R}]=\emptyset)\right)
 =
 \lambda\mathbb{E}\left(\frac{1}{Z_\mathbf{T}(\lambda)}\right).
$$

Finally, the function $x\mapsto e^{-x}$ is convex, so Jensen's
inequality gives
$$
 \mathbb{E}\left(\frac1{Z_\mathbf{T}(\lambda)}\right)
 =
 \mathbb{E} e^{-\log Z_\mathbf{T}(\lambda)}
 \geq
 e^{-\mathbb{E}\log Z_\mathbf{T}(\lambda)}
 =
 e^{-z}.
$$
Substituting into \eqref{eq8} proves
\eqref{eq9}.
\end{proof}
\begin{lemma}\label{secondlocal}
For the random set $\mathbf{T}=\mathbf{T}(\mathbf{R},\mathbf{Y})$,
\begin{equation}\label{eq10}
  \tilde{\rho}\geq 3^{-d}\mathbb{E}\rho_\mathbf{T}(\lambda).
\end{equation}
\end{lemma}

\begin{proof}
By the definition of $\rho_\mathbf{T}(\lambda)$ and Lemma \ref{lem:spatial-markov}, we have
$$
 \rho_\mathbf{T}(\lambda)
 =
 \mathbb{E}\left(|\mathbf{X}\cap N[\mathbf{R}]||\mathbf{R},\mathbf{Y}\right).
$$
Taking expectations,
$$
 \mathbb{E}\rho_\mathbf{T}(\lambda)
 =
 \mathbb{E}|\mathbf{X}\cap N[\mathbf{R}]|.
$$
Using the distribution of $\mathbf{R}$ and then exchanging the two finite
sums, we obtain
\begin{align*}
 \mathbb{E}|\mathbf{X}\cap N[\mathbf{R}]|
 &=
 \frac{1}{M}
 \mathbb{E}_\lambda\left(\sum_{v_i\in V}a_i|\mathbf{X}\cap N[v_i]|\right)\\
 &=
 \frac{1}{M}
 \mathbb{E}_\lambda\left(\sum_{v_i\in V}\sum_{\substack{v\in \mathbf{X}\\v\in N[v_i]}}a_i\right)\\
 &=
 \frac{1}{M}
 \mathbb{E}_\lambda\left(\sum_{v\in \mathbf{X}}\sum_{\substack{v_i\in V\\v_i\in N[v]}}a_i\right),
\end{align*}
where the last equality follows from the fact that adjacency is symmetric.

Recalling the definition of $m_v$ and Lemma \ref{mxbound}, we conclude that
\begin{align*}
 \mathbb{E}\rho_\mathbf{T}(\lambda)
 &=\mathbb{E}|\mathbf{X}\cap N[\mathbf{R}]|\\
 &=\frac{1}{M}\mathbb{E}_\lambda\left(\sum_{v\in \mathbf{X}}\sum_{\substack{v_i\in V\\v_i\in N[v]}}a_i\right)\\
 &=\frac{1}{M}\mathbb{E}_\lambda\left(\sum_{v\in \mathbf{X}}m_v\right)\\
 &\leq\frac{1}{M}\mathbb{E}_\lambda\left(\sum_{v\in \mathbf{X}}3^d\right)\\
 &=\frac{3^d}{M}\mathbb{E}_\lambda|\mathbf{X}|=3^d\tilde{\rho}.
\end{align*}
Rearranging proves \eqref{eq10}.
\end{proof}

Set
$$
 \sigma=1+\sqrt3.
$$
We have the following geometric estimate.
\begin{lemma}[Weighted lens lemma]\label{lem:weighted-lens}
Let $x_1, x_2, \ldots, x_n\in B^d(0,2)$. Suppose that $C_1, C_2, \ldots, C_n$ are pairwise disjoint measurable sets satisfying
$$
 C_i\subseteq B^d(x_i,1).
$$
Define
$$
 c_i=\frac{\operatorname{vol}(C_i)}{\omega_d},
 \qquad
 g=\sum_{i=1}^nc_i,
$$
and
$$
 g_i
 =
 \sum_{\substack{1\leq j\leq n\\\|x_i-x_j\|\leq2}}c_j.
$$
If $g>0$, then
$$
 \frac{\sum_{i=1}^nc_ig_i}{g}
 \leq
 2(1+\sqrt3)^d
 =
 2\sigma^d.
$$
\end{lemma}

\begin{proof}
Put
$$
 r_i=\|x_i\|
$$
and
\begin{equation}\label{eq11}
 D
 =
 \sum_{i=1}^nc_i
 \sum_{\substack{1\leq j\leq n\\
                 \|x_i-x_j\|\leq2\\
                 r_j\leq r_i}}
 c_j.
\end{equation}

We first claim that
\begin{equation}\label{eq12}
  \sum_{i=1}^nc_ig_i\le2D.
\end{equation}
Indeed, the diagonal terms contribute $\sum_{i=1}^nc_i^2$ to the left-hand side and $2\sum_{i=1}^nc_i^2$ to the right-hand side, since
\(r_i\le r_i\). For distinct $i,j$ with $\|x_i-x_j\|\leq2$, the two ordered pairs $(i, j)$ and $(j, i)$ contribute a total of $2c_ic_j$ to $\sum_{i=1}^n c_ig_i$. Since either
$r_j\leq r_i$ or $r_i\leq r_j$, the quantity $D$ contains at least one contribution $c_ic_j$. Thus the corresponding contribution to $2D$ dominates these two ordered terms. This proves \eqref{eq12}.

For a point $x\in B^d(0,2)$, let $r=\|x\|$ and define a closed lens
$$
 L(x, r):=B^d(x,2)\cap B^d(0,r).
$$
If $x_i$ and $x_j$ satisfy $\|x_i-x_j\|\leq2$ and $r_j\leq r_i$, then $x_j\in L(x_i, r_i)$. Since $C_j\subseteq B^d(x_j,1)$, we have
$$
 \bigcup_{\substack{j:\|x_i-x_j\|\leq2\\r_j\le r_i}}C_j
 \subseteq
 L(x_i,r_i)+B^d(0,1),
$$
where $+$ denotes Minkowski addition. Pairwise disjointness of the sets $C_j$ therefore gives
$$
 \sum_{\substack{j:\|x_i-x_j\|\leq2\\r_j\le r_i}}c_j
 \leq
 \frac{\operatorname{vol}\left(L(x_i,r_i)+B^d(0,1)\right)}{\omega_d}.
$$

For every $x\in B^d(0,2)$ with norm $r$, the lens $L(x, r)$ is contained in a ball of radius $\sqrt3$; see, for example, the proof of \cite[Lemma 10]{MR3898718}. Thus for every $i$,
$$
\operatorname{vol}\left(L(x_i,r_i)+B^d(0,1)\right)\leq(1+\sqrt3)^d\omega_d=\sigma^d\omega_d,
$$
and
$$
 \sum_{\substack{j:\|x_i-x_j\|\leq2\\r_j\le r_i}}c_j
 \leq
 \sigma^d.
$$
Substituting into \eqref{eq11} gives
$$
 D\leq\sigma^d\sum_{i=1}^nc_i=\sigma^dg.
$$
Finally, \eqref{eq12} yields
$$
 \sum_{i=1}^nc_ig_i
 \leq
 2D
 \leq
 2\sigma^dg.
$$
Dividing by $g$ proves the lemma.
\end{proof}

\begin{lemma}\label{lem:local-partition}
For every realization $T$ of $\mathbf{T}(\mathbf{R},\mathbf{Y})$, we have
$$
 \rho_T(\lambda)
 \geq
 e^{-2\lambda\sigma^d}\log Z_T(\lambda).
$$
\end{lemma}

\begin{proof}
Let $T$ be a realization of $\mathbf{T}(\mathbf{R},\mathbf{Y})$, arising from $\mathbf{X}=I\in\mathcal{I}(G)$ and $\mathbf{R}=v\in V$.
Put
$$
 H=G[T].
$$
For $v_i\in T$, define
$$
 m_{H, i}
 =
 \sum_{v_j\in N_H[v_i]}a_j,
$$
and let
$$
 m_H=\sum_{v_i\in T}a_i.
$$
Since $v\in T$ and all vertex weights are positive, we have $m_H>0$.

Let $\mathbf{J}$ be an independent set drawn from the weighted hard-core
model on $H$, and define
$$
 b_{H, i}
 =
 \mathbb{P}_H(\mathbf{J}\cap N_H[v_i]=\emptyset),
$$
where $\mathbb{P}_H$ is the probability distribution of the weighted hard-core
model on $H$. The argument used in the proof of Lemma \ref{lem:first-local} shows that
$$
 \mathbb{P}_H(v_i\in \mathbf{J})
 =
 \lambda a_ib_{H, i},
$$
and
\begin{equation}\label{eq13}
 \rho_T(\lambda)
 =
 \lambda\sum_{v_i\in T}a_ib_{H, i}.
\end{equation}

We claim that
\begin{equation}\label{eq14}
  b_{H, i}\geq e^{-\lambda m_{H, i}}.
\end{equation}
To see this, conditioned on the configuration outside $N_H[v_i]$, the available vertices inside $N_H[v_i]$ form
some set $K_i\subseteq N_H[v_i]$. By Lemma \ref{lem:spatial-markov}, the conditional probability that
none of the vertices in $K_i$ is occupied equals
$$
 \frac{1}{Z_{K_i}(\lambda)}.
$$
Since the weighted partition function is bounded by the weighted generating
function of all subsets, we have
\begin{align*}
 Z_{K_i}(\lambda)&=\sum_{J\in\mathcal{I}(H[K_i])}\prod_{v_j\in J}\left(\lambda a_j\right)\\
 &\leq
 \sum_{S\subseteq K_i}\prod_{v_j\in S}\left(\lambda a_j\right)\\
 &=
 \prod_{v_j\in K_i}(1+\lambda a_j)\\
 &\leq
 \prod_{v_j\in K_i}\exp\left(\lambda a_j\right)\\
 &=
 \exp\left(\sum_{v_j\in K_i}\lambda a_j\right)\\
 &\le
 e^{\lambda m_{H, i}}.
\end{align*}
Thus the conditional probability that
none of the vertices in $K_i$ is occupied is at least
\(e^{-\lambda m_{H, i}}\). Averaging over the exterior configuration
proves \eqref{eq14}.

Substituting \eqref{eq14} into
\eqref{eq13} gives
$$
 \rho_T(\lambda)
 \ge
 \lambda\sum_{v_i\in T}a_ie^{-\lambda m_{H, i}}.
$$
Applying Jensen's inequality to the convex function
$x\mapsto e^{-\lambda x}$, with normalized weights $a_i/m_H$, we obtain
\begin{align}\label{eq15}
 \rho_T(\lambda)
 &\geq
 \lambda m_H
 \sum_{v_i\in T}\frac{a_i}{m_H}e^{-\lambda m_{H, i}}\nonumber\\
 &\ge
 \lambda m_H
 \exp\left(
   -\lambda
   \frac{\sum_{v_i\in T}a_im_{H, i}}{m_H}
 \right).
\end{align}

Since $T\subseteq N[v]$, we have $\|v_i-v\|\leq2$ for every $v_i\in T$. Lemma~\ref{lem:weighted-lens} therefore gives
$$
 \frac{\sum_{v_i\in T}a_im_{H, i}}{m_H}
 \leq
 2\sigma^d.
$$
Substituting into \eqref{eq15} yields
\begin{equation}\label{eq16}
\rho_T(\lambda)
 \geq
 \lambda m_H e^{-2\lambda\sigma^d}.
\end{equation}

Finally,
\begin{align*}
 Z_T(\lambda)
 &\leq
 \prod_{v_i\in T}(1+\lambda a_i)\\
 &\leq
 \exp\left(\lambda\sum_{v_i\in T}a_i\right)\\
 &=
 e^{\lambda m_H}.
\end{align*}
Hence
\[
 \log Z_T(\lambda)\le\lambda m_H.
\]
Combining this with \eqref{eq16} proves
$$
 \rho_T(\lambda)
 \geq
 e^{-2\lambda\sigma^d}\log Z_T(\lambda).
$$
\end{proof}

We are now ready to prove Theorem \ref{qiang}.
\begin{proof}[Proof of Theorem \ref{qiang}]
Recall that in Lemma \ref{lem:first-local},
$$
 z=\mathbb{E}\log Z_T(\lambda).
$$
Taking expectations in Lemma~\ref{lem:local-partition} and applying
Lemma~\ref{secondlocal}, we obtain
$$
 \tilde{\rho}
 \geq
 3^{-d}e^{-2\lambda\sigma^d}z.
$$
Lemma~\ref{lem:first-local} also gives
$$
 \tilde{\rho}\geq\lambda e^{-z}.
$$
Consequently,
$$
 \tilde{\rho}
 \geq
 \inf_{z\geq0}\max\left\{
   \lambda e^{-z},
   \,
   3^{-d}e^{-2\lambda\sigma^d}z
 \right\}.
$$

The function \(z\mapsto\lambda e^{-z}\) is decreasing, while
\(z\mapsto 3^{-d}e^{-2\lambda\sigma^d}z\) is increasing. Therefore the infimum over \(z\ge0\)
of their maximum occurs at their unique intersection, and $\tilde{\rho}\geq 3^{-d}e^{-2\lambda\sigma^d}z^*$, where
$$
 \lambda e^{-z^*}=3^{-d}e^{-2\lambda\sigma^d}z^*.
$$
Equivalently,
$$
 z^*e^{z^*}=\lambda3^{d}e^{2\lambda\sigma^d}.
$$
So
$$
 z^*=W\left(\lambda3^{d}e^{2\lambda\sigma^d}\right)\text{ and }\tilde{\rho}\geq 3^{-d}e^{-2\lambda\sigma^d}\cdot W\left(\lambda3^{d}e^{2\lambda\sigma^d}\right),
$$
where $W$ is the Lambert $W$-function.

Taking $\lambda=\sigma^{-d}/d$, we have $2\lambda\sigma^d=2/d$ and $\lambda3^de^{2\lambda\sigma^d}=\frac{(3/\sigma)^d}{d}e^{2/d}$.
Since $3/\sigma=3/(1+\sqrt3)>1$, $\frac{(3/\sigma)^d}{d}e^{2/d}$ tends to infinity exponentially in $d$.

For $x\to\infty$,
$$
 W(x)=\log x-\log\log x
      +O\left(\frac{\log\log x}{\log x}\right).
$$
We calculate
$$
 W\left(\lambda3^{d}e^{2\lambda\sigma^d}\right)
 =
 d\log\frac{3}{\sigma}-2\log d+O(1).
$$
Therefore,
\begin{align}\label{eq18}
 \tilde{\rho}
 &\geq
 3^{-d}e^{-2/d}
 \left(
   d\log\frac3\sigma-2\log d+O(1)
 \right)\nonumber\\
 &=
 \left(
   \log\frac3\sigma
   -O\left(\frac{\log d}{d}\right)
 \right)d\,3^{-d}.
\end{align}

Let $\mathcal{C}$ be an arbitrary nonempty finite collection of closed
unit balls. By \eqref{eqp} and
\eqref{eq18},
$$
 \frac{\alpha(\mathcal{C})\omega_d}{\operatorname{vol}(U)}
 \ge
 \left(
   \log\frac3\sigma
   -O\left(\frac{\log d}{d}\right)
 \right)d\,3^{-d},
$$
where $\sigma=1+\sqrt3$. Since the bound is uniform in
$\mathcal{C}$, it follows that
$$
 f(B^d)
 \ge
 \left(
   \log\frac{3}{1+\sqrt3}
   -O\left(\frac{\log d}{d}\right)
 \right)d\,3^{-d},
$$
where the implicit constant in the $O$-term is absolute and independent of the collection $\mathcal{C}$.
\end{proof}

\section{Concluding remarks}
For a convex body $K$, let $\delta(K)$ denote its translative packing density. The known values and lower bounds are summarized in Table \ref{table1}. 
\begin{table*}[b]
\renewcommand\arraystretch{1.5}
    \caption{Known bounds for $\delta(K)$ and $f(K)$}
    \centering
    \begin{tabular}{|c|c|c|}
        \hline
        Types of convex bodies &$Q^d$ & $B^d$\\
        \hline

        Known bounds for $\delta(K)$&  $\delta(Q^d)=1$ & $\delta(B^d)=\Omega(d2^{-d})$ \cite{MR3898718}\\\hline

        Known bounds for $f(K)$& $f(Q^d)=2^{-d}$ \cite{MR30782}& $f(B^d)=\Omega(d3^{-d})$ (Theorem \ref{qiang})  \\
        \hline
    \end{tabular}
    \label{table1}
\end{table*}
Motivated by Table \ref{table1}, we propose the following conjecture, which relates $\delta(K)$ and $f(K)$.
\begin{conjecture}\label{kunnan}
Let $\{K_d\}_{d\geq1}$ be a family of convex bodies, where $K_d\subseteq\mathbb{R}^d$ for every $d\geq1$. If there exist absolute constants $C>0, k\geq0$, and $a\geq1$, independent of $d$, such that
$$
\delta(K_d)\geq Cd^ka^{-d}
$$
for every $d\geq1$, then there exists an absolute constant $C'>0$ such that
$$
f(K_d)\geq C'd^k(1+a)^{-d}
$$
for every $d\geq1$.
\end{conjecture}

\bibliographystyle{abbrv}
\bibliography{radocovering_REF}
\end{document}